\documentclass[a4paper,11pt]{amsart}

\usepackage{amsmath}
\usepackage{amsfonts}
\usepackage{amssymb}
\usepackage{graphicx}

\newtheorem{theorem}{Theorem}
\newtheorem{claim}[theorem]{Claim}

\newtheorem{lemma}[theorem]{Lemma}
\newtheorem{proposition}[theorem]{Proposition}

\theoremstyle{definition}
\newtheorem{definition}[theorem]{Definition}

\theoremstyle{remark}

\numberwithin{theorem}{section}
\numberwithin{equation}{section}

\begin{document}
\title[Directional Lower Porosity]{Directional Lower Porosity}
\author{Gareth Speight}
\email{G.Speight@Warwick.ac.uk}

\thanks{This work was done while the author was a PhD student of David Preiss and supported by EPSRC funding. I thank David Preiss and Ludek Zaj\'{i}\v{c}ek for suggesting the question which led to this paper.}

\begin{abstract}
We investigate differences between upper and lower porosity. In finite dimensional Banach spaces every upper porous set is directionally upper porous. We show the situation is very different for lower porous sets; there exists a lower porous set in $\mathbb{R}^2$ which is not even a countable union of directionally lower porous sets.
\end{abstract}

\maketitle

\section{Introduction}

There are two main types of porosity of sets in metric spaces. A set is upper (lower) porous if for each point of the set there are nearby holes in the set, of radius proportional to their distance away, at arbitrarily small (all sufficiently small) scales. A set is $\sigma$-upper (lower) porous if it is a countable union of upper (lower) porous sets.

Despite the similarity of the definitions, upper and lower porous sets can behave very differently; in any complete metric space with no isolated points, there exists a closed set which is upper porous but not $\sigma$-lower porous (Remark 2.8(ii) \cite{Zaj05}).

Upper porosity is often used to estimate the size of exceptional sets arising in differentiation theory or other areas of classical analysis. Proving an exceptional set is ($\sigma$-)upper porous often gives a stronger result than merely showing it is small in the sense of category or measure. Indeed, in any complete metric space with no isolated points, there exists a closed nowhere dense set which is not $\sigma$-upper porous \cite{Zaj98}. Further, in $\mathbb{R}^{n}$, there exists a closed nowhere dense set of Lebesgue measure zero which is not $\sigma$-upper porous \cite{Zaj88}. 

The investigation of $\sigma$-upper porous sets was started by Dol\v{z}enko \cite{Dol67} in 1967. He proved that particular exceptional sets arising in the theory of cluster sets are $\sigma$-upper porous. One of the first applications of $\sigma$-upper porous sets to differentiation theory was in the study of the symmetric derivative. Suppose $f \colon \mathbb{R} \to \mathbb{R}$ is continuous at a dense set of points. Then, except at a set of points which is $\sigma$-upper porous, the upper symmetric derivative is the maximum of the left upper and right upper Dini derivatives \cite{Beh78}. This then implies the set of points at which $f$ is symmetrically differentiable but not differentiable is $\sigma$-upper porous.

Lower porosity is perhaps best known for its implications for the Hausdorff dimension of a set. For example, it is known that if a set in $\mathbb{R}^{n}$ is lower porous and the size of holes is close to maximal then the Hausdorff dimension of the set can be only slightly above $n-1$ \cite{Mat88}. However, as pointed out by Zaj\'{i}\v{c}ek, lower porosity also has some applications in differentiation theory. For example, suppose $X$ is an Asplund space and $f \colon X \to \mathbb{R}$ is a continuous convex function. Then it follows from results in \cite{Zaj91} that the set of points where $f$ is not Fr\'{e}chet differentiable is $\sigma$-lower porous. 

The survey papers \cite{Zaj05} and \cite{Zaj88} discuss many more applications of upper and lower porosity. Since proving an exceptional set is ($\sigma$-)lower porous, rather than ($\sigma$-)upper porous, gives a stronger result, we would like to understand what properties of upper porosity which are useful in this context remain valid for lower porosity. In particular, we investigate the relation between porosity and directional porosity. 

A set in a Banach space is called directionally upper/lower porous if at each point of the set the corresponding holes lie in a fixed direction. In finite dimensional Banach spaces compactness of the unit sphere implies every upper porous set is directionally upper porous. We show the situation is very different for lower porous sets; there exists a lower porous set in $\mathbb{R}^2$ which is not even a countable union of directionally lower porous sets.

\section{Basic notions}

We now give relevant definitions and establish a result we will use to prove a set is not a countable union of directionally lower porous sets. 

If $x$ is a point in a metric space and $r>0$ we denote by $B(x,r)$ the open ball of centre $x$ and radius $r$.

\begin{definition}\label{def}
Let $M$ be a metric space. We say $P\subset M$ is lower porous at $x \in P$ if there exists $\rho>0$ and $r_{0}>0$ such that for every $0<r<r_{0}$ there exists $y \in M$ with $d(x,y)<r$ such that
\[B(y,\rho r)\cap P = \varnothing.\]

Let $X$ be a Banach space. We say $P \subset X$ is lower porous at $x \in P$ in direction $v \in X$ if the points $y$ in the above definition can be chosen on the line through $x$ in direction $v$.

We say $P$ is (directionally) lower porous if it is lower porous at each of its points (in some direction).

We say a set is $\sigma$-(directionally) lower porous if it is a countable union of (directionally) lower porous sets.
\end{definition}

We will informally refer to the balls $B(y,\rho r)$ in Definition \ref{def} as holes in $P$. Note by replacing `for every $0<r<r_{0}$' in Definition \ref{def} with `there exists $0<r<r_{0}$' one defines the corresponding notions of upper porosity.

While it can be difficult to prove a set is not $\sigma$-(directionally) upper porous the situation for lower porosity is simpler. Variants of the following lemma and proposition are stated in a survey paper by Zaj\'{i}\v{c}ek (Proposition 2.5, Proposition 2.6 \cite{Zaj05}) for lower porous sets in metric spaces. We reformulate them for directionally lower porous sets in separable Banach spaces. The proofs are almost the same.

\begin{lemma}\label{closedcover}
Let $A$ be a $\sigma$-directionally lower porous subset of a separable Banach space $X$. Then $A$ can be covered by countably many closed directionally lower porous sets.
\end{lemma}

\begin{proof}
It is easy to see one can write $A=\bigcup_{n=1}^{\infty} P_{n}$ where each set $P_{n}$ has the following property for some fixed $v_{n} \in X$, $\rho_{n}>0$ and $r_{n}>0$:
\[ \forall \, x \in P_{n}\, \forall \, 0<r<r_{n}\, \exists \, -r<t<r \,:\, B(x+tv_{n},\rho_{n} r)\cap P_{n} = \varnothing.\]
One can show by an approximation argument that the sets $\overline{P_{n}}$ are directionally lower porous (possibly with slightly smaller holes). Since $A\subset\bigcup_{n=1}^{\infty} \overline{P_{n}}$ we see $A$ can be covered by countably many closed directionally lower porous sets.
\end{proof}

\begin{proposition}\label{sufficient}
Let $X$ be a separable Banach space and let $F$ be a closed subset of $X$. Suppose there exists $A \subset F$ which is dense in $F$ such that $F$ is not directionally lower porous (in $X$) at any point of $A$. Then $F$ is not $\sigma$-directionally lower porous.
\end{proposition}

\begin{proof}
Suppose $F$ is $\sigma$-directionally lower porous. Then by Lemma \ref{closedcover} we can write $F \subset \bigcup_{n=1}^{\infty} P_{n}$ where each set $P_{n}$ is closed and directionally lower porous. 

Since $F$ is complete, by Baire's theorem one of the sets $P_{N}$ cannot be nowhere dense in $F$. Since the set $P_{N}$ is closed it follows there is an open set $U \subset X$ such that $\varnothing \neq U \cap F \subset P_{N}$. 

Now choose $y \in A \cap U$. Since $P_{N}$ is directionally lower porous it follows that there are holes close to $y$, in a fixed direction and at all sufficiently small scales, which avoid $P_{N}$. Since $U \cap F \subset P_{N}$ it follows, on all sufficiently small scales, these holes avoid $F$ also. This implies $F$ is directionally lower porous at $y$ which contradicts $y \in A$.
\end{proof}

\section{Construction of the lower porous set}

We now construct a set in $\mathbb{R}^2$ which is lower porous but not $\sigma$-directionally lower porous.

Suppose $A \subset \mathbb{R}^{2}$ and $x \in \mathbb{R}^{2}$. Then $\partial A$ denotes the boundary of $A$, $\overline{A}$ denotes the closure of $A$, and $\| x\|$ is the Euclidean norm of $x$. If $A$ is non empty we also define,
\[d(x,A)=\inf \{\|x-y\|:y \in A\}\]
and, for $r>0$,
\[B(A,r)=\{x \in \mathbb{R}^2: d(x,A)<r\}.\]
Let $B(\varnothing,r)=\varnothing$.

If $v \in \mathbb{R}^2$ we denote by $(v_{1},v_{2})$ the coordinates of $v$ and let $v^{\bot}=(-v_{2},v_{1})$ be the anticlockwise rotation of $v$ through a right angle.

The lower porous set will be constructed by removing neighbourhoods of lines in different directions. We now define an appropriate sequence of directions and give the details of the construction.

\begin{definition}\label{sequence}
Fix a sequence $(v_{n})_{n=1}^{\infty} \subset S^{1}=\{v \in \mathbb{R}^2:\|v\|=1\}$ with the following properties:
\begin{itemize}
	\item The set $D=\{v_{n}: n \in \mathbb{N}\}$ is dense in $S^{1}$.
	\item The sequence $(v_{n})_{n=1}^{\infty}$ contains arbitrarily long strings of each of its terms. More precisely, for each $v \in D$ and $N\geq 1$ there exists $n\geq N$ such that
	\[v_{n+1}=v_{n+2}=\ldots =v_{n+N}=v.\]
\end{itemize}
\end{definition}

\begin{definition}\label{set}
For $n \in \mathbb{N}$ and $k \in \mathbb{Z}$ let $L_{n}^{k}$ denote the line in direction $v_{n}$ translated by the vector $(k/2^{6n})v^{\bot}$ from the origin.

We define $H_{0}=\varnothing$ and, for $n \geq 1$,
\[C_{n}=\bigcup_{k=-\infty}^{\infty} L_{n}^{k} \setminus B(H_{n-1},1/2^{6n}),\]
\[H_{n}=H_{n-1} \cup \bigcup_{x \in C_{n}} B(x,1/2^{6(n+1)}).\]
Finally we define
\[H=\bigcup_{n=1}^{\infty} H_{n}\]
and
\[P=\mathbb{R}^2 \setminus H.\]
\end{definition}

Intuitively $P$ corresponds to a construction similar to that of the Cantor set in which the points removed are close to lines which rotate throughout the construction. At each stage of the construction we take care not to remove points too close to those previously removed.

The set $P$ should be lower porous because every point of $P$ will see a nearby line with removed points at every scale. 

At many points $P$ should not be directionally lower porous because, for each direction, there are arbitrarily long sequences of scales in which holes are constructed only centred on well separated lines in that direction.

We now prove formally that $P$ is lower porous but not $\sigma$-directionally lower porous.

\section{The set $P$ is lower porous}

To show $P$ is lower porous we show that, for $n\geq 1$, each point of $P$ is relatively close to $\partial H_{n}$ and then that close to $\partial H_{n}$ there is a large ball which is disjoint from $P$. We now prove several lemmas which make this argument precise.

\begin{lemma}\label{bdry}
The sequence of sets $\partial H_{n}$ is increasing and contained inside $P$. In particular, $P \neq \varnothing$.
\end{lemma}

\begin{proof}
Suppose $x \in \partial H_{n}$. Then every neighbourhood of $x$ meets $H_{n}$ and hence, since $H_{n} \subset H_{n+1}$, every neighbourhood of $x$ meets $H_{n+1}$. Since $d(x,H_{n+1}\setminus H_{n})>0$ and every neighbourhood of $x$ meets $\mathbb{R}^2 \setminus H_{n}$ it follows that every neighbourhood of $x$ meets $\mathbb{R}^2\setminus H_{n+1}$. Hence $x \in \partial H_{n+1}$.

If $x \in \partial H_{n}$ for some $n \geq 1$ then $x\notin H_{m}$ for any $m\geq 1$. Consequently $x \in P$.
\end{proof}

\begin{lemma}\label{close}
Suppose $x \notin H_{n}$. Then there exists a point $z \in \partial H_{n}$ such that $\|x-z\|<1/2^{6n-1}$.
\end{lemma}

\begin{proof}
There exists $y \in \bigcup_{k=-\infty}^{\infty} L_{n}^{k}$ with $\|x-y\|\leq 1/2^{6n+1}$. 

If $y\notin B(H_{n-1},1/2^{6n})$ then $B(y,1/2^{6(n+1)})\subset H_{n}$ and there is $z \in \partial H_{n}$ with $\|y-z\|= 1/2^{6(n+1)}$. In particular $z \in \partial H_{n}$ and $\|x-z\|<1/2^{6n}$.

If $y \in B(H_{n-1},1/2^{6n}) \setminus H_{n-1}$ it is clear there exists $z \in \partial H_{n-1}$ with $\|y-z\|< 1/2^{6n}$. Hence, by Lemma \ref{bdry}, $z \in \partial H_{n}$ and $\|x-z\|<1/2^{6n-1}$.

If $y \in H_{n-1}$ then, since $x\notin H_{n-1}$, there is $z \in \partial H_{n-1}\subset \partial H_{n}$ on the line segment joining $x$ and $y$. In this case $\|x-z\|\leq 1/2^{6n+1}$.
\end{proof}

\begin{lemma}\label{thick}
Suppose $x \in \overline{H_{n}}$. Then there exists a point $y \in \mathbb{R}^2$ such that $\|x-y\|\leq 1/2^{6(n+1)}$ and $B(y,1/2^{6(n+1)})\subset H_{n}$.
\end{lemma}

\begin{proof}
The lemma follows from the fact $H_{n}$ is a union of balls of radii at least $1/2^{6(n+1)}$.
\end{proof}

We now combine the previous lemmas to show $P$ is lower porous.

\begin{theorem}\label{lp}
The set $P$ is lower porous.
\end{theorem}

\begin{proof}
Let $x \in P$ and $n\geq 1$. Then $x\notin H_{n}$ so by Lemma \ref{close} there exists $y \in \partial H_{n}$ with $\|x-y\|<1/2^{6n-1}$. 

By Lemma \ref{thick} there is $z \in \mathbb{R}^2$ with $\|y-z\|\leq 1/2^{6(n+1)}$ such that $B(z,1/2^{6(n+1)}) \subset H_{n}$.

It follows $\|x-z\|<1/2^{6n-2}$ and $B(z,1/2^{6(n+1)}) \cap P = \varnothing$. Thus for each $n\geq 1$ there is a hole relatively close to $x$ of radius $1/2^{6(n+1)}$.

For sufficiently small $r>0$ there is $n\geq 1$ such that $1/2^{6(n+1)}\leq r \leq 1/2^{6n}$. Hence for any sufficiently small scale we can find a relatively large hole which is close to $x$. It follows $P$ is lower porous at $x$.
\end{proof}

\section{The set $P$ is not $\sigma$-directionally lower porous}

Notice $P$ is a closed set. To show $P$ is not $\sigma$-directionally lower porous it suffices, by Proposition \ref{sufficient}, to show $P$ is not directionally lower porous at points of a dense set. We now construct a dense set consisting of points with useful properties and then show $P$ is not lower porous at such points.

Given $s \geq 1$ define
\[A_{s}= \{ x \in P: \overline{B(x,1/2^{6s+5})}\cap \overline{B(H_{s},1/2^{6(s+1)})}=\varnothing \}.\]

\begin{lemma}\label{antihole}
Suppose $w \in P$ and $n\geq 1$. Then for sufficiently large $s \geq 1$ there exists $z \in A_{s}$ with $\|w-z\|<1/2^{6(n-1)}$.
\end{lemma}

\begin{proof}
Since $w \in P$ implies $w\notin H_{n}$ we can find, using Lemma \ref{close}, $x \in \partial H_{n}$ with $\|w-x\|<1/2^{6n-1}$.
Choose $1\leq m \leq n$ such that $x \in \partial H_{m} \setminus \partial H_{m-1}$. By moving $x$ slightly if necessary we may assume that points in $H_{m}$ that are sufficiently close to $x$ lie in a half plane whose boundary meets $x$.

For sufficiently large $s$ we can find $y \in \mathbb{R}^2$ with $\|x-y\|<1/2^{6s+2}$ such that
\[B(y,1/2^{6s+3}) \cap H_{m}=\varnothing.\]
By Definition \ref{set} we see, if $m<s$,
\[d(x,H_{s}\setminus H_{m})\geq 1/2^{6s}-1/2^{6(s+1)} \geq 1/2^{6s+1}.\]
This inequality implies
\[B(y,1/2^{6s+3}) \cap H_{s}=\varnothing.\]
Clearly then
\[B(y,1/2^{6s+4}) \cap \overline{B(H_{s},1/2^{6(s+1)})}=\varnothing.\]
Using Definition \ref{set} we can find $z \in \partial H_{s+1} \subset P$ with $\|y-z\|\leq 1/2^{6(s+1)}$. It follows
\[\overline{B(z,1/2^{6s+5})}\cap \overline{B(H_{s},1/2^{6(s+1)})}=\varnothing\]
so $z \in A_{s}$.

By the triangle inequality we observe $\|w-z\|<1/2^{6(n-1)}$.
\end{proof}

For $v \in D$ and $N\geq 1$, let $G_{v,N}$ be the set of $x$ for which there is $n\geq N$ with $x \in A_{n}$ and $v_{n+1}=\ldots = v_{n+N}=v$.

\begin{proposition}\label{goodset}
For each $v \in D$ and $N\geq 1$ the set $G_{v,N}$ is dense and open in $P$.
\end{proposition}

\begin{proof}
For each $n\geq 1$ the set $A_{n}$ is open; hence $G_{v,N}$, as a union of open sets, is open. That $G_{v,N}$ is dense follows from Lemma \ref{antihole} because, by Definition \ref{sequence}, there is arbitrarily large $n\geq N$ such that $v_{n+1}=\ldots = v_{n+N}=v$.
\end{proof}

\begin{proposition}\label{ndlp}
Suppose $v \in D$. The set $P$ is not directionally lower porous in direction $v$ at points of the set $\bigcap_{N=1}^{\infty} G_{v,N}$.
\end{proposition}

\begin{proof}
To simplify notation suppose, without loss of generality, $v=(0,1)$. The proof for the general case is the same up to a rotation.

Suppose $x \in \bigcap_{N=1}^{\infty} G_{v,N}$ and $P$ is directionally lower porous at $x$ in direction $v$. 

By Definition \ref{def} there exists $\rho>0$ and $r_{0}>0$ such that for every $0<r<r_{0}$ there exists $-r<t<r$ such that
\begin{equation}
B(x+tv,\rho r)\cap P=\varnothing.
\label{rho}
\end{equation}

Choose $N$ such that $1/2^{6N}<r_{0}$. Then, as $x \in G_{v,N}$, there exists $n\geq N$ such that $v_{n+1}=\ldots = v_{n+N}=v$ and $x \in A_{n}$. 

We now show if $t$ is relatively small then $x+tv\notin H_{n+N}$ and hence there are points of $P$ close to $x+tv$. To do this we analyse the structure of $H_{n+N}$ close to $x$.

Choose $s_{1}<0<s_{2}$, with $|s_{1}|$ and $|s_{2}|$ minimal, such that
\[x+(s_{1},0) \in L_{n+1}^{k_{1}}\]
and
\[x+(s_{2},0) \in L_{n+1}^{k_{2}}\]
for some $k_{1},k_{2} \in \mathbb{Z}$. 

Notice, by Definition \ref{set}, it follows $|s_{1}|, |s_{2}|\leq 1/2^{6(n+1)}$. 

Hence if
\[R=[x_{1}+s_{1},x_{1}+s_{2}]\times [x_{2}-1/2^{6(n+1)},x_{2}+1/2^{6(n+1)}]\]
then
\[R \subset B(x,1/2^{6n+5})\]
so, since $x \in A_{n}$,
\[R \cap \overline{B(H_{n},1/2^{6(n+1)})}=\varnothing.\]

The following claim gives information about $H_{n+N}$ which we will need to finish the proof.

\begin{claim}
Fix $n+1\leq m\leq n+N$. Then the following statements hold:
\begin{itemize}
	\item If $y, y+tv \in R$ then $y \in H_{m}$ if and only if $y+tv \in H_{m}$.
	\item If $z \in H_{m}\setminus R$ and $z+tv \in R\setminus H_{m}$ for some $t \in \mathbb{R}$ then
	\[d(z,R)\geq 1/2^{6m}-1/2^{6(m+1)}.\]
\end{itemize}
\end{claim}

\begin{proof}[Proof of claim]
Notice, by Definition \ref{set} and the definition of $R$, that if $y \in H_{n+N}$ with $x_{1}+s_{1} \leq y_{1} \leq x_{1}+s_{2}$ then $y \in B(z,1/2^{6(m+1)})$ for some $n+1\leq m \leq n+N$ and $z\in C_{m}$ on a line $L_{m}^{k}$ that meets $R$. Thus it suffices to analyse which points on lines $L_{m}^{k}$ meeting $R$ lie in $C_{m}$.

We prove the claim by induction with respect to $m$. 

Let $m=n+1$. In this case, by the definition of $R$,
\[H_{m}\cap R=\left(B(L_{n+1}^{k_{1}},1/2^{6(n+2)}) \cup B(L_{n+1}^{k_{2}},1/2^{6(n+2)})\right) \cap R\]
so the first part of the claim is clearly true. It follows from Definition \ref{set} and the definition of $R$ that there are no points $z \in H_{n+1}\setminus H_{n}$ satisfying the assumptions in the second part of the claim. Since
\[R \cap \overline{B(H_{n},1/2^{6(n+1)})}=\varnothing\]
it follows the second part of the claim holds for $m=n+1$.

Suppose the claim holds for some $n+1\leq m\leq n+N-1$; we show it holds with $m$ replaced by $m+1$. First notice $v_{m+1}=v$. Suppose $x+(s,0) \in L_{m+1}^{k}$ for some $s_{1}<s<s_{2}$ and $k \in \mathbb{Z}$. 

If $x+(s,0) \in \overline{H_{m}}$ then, by the first part of the claim for $m$, $x+(s,t) \in \overline{H_{m}}$ for all $|t|\leq 1/2^{6(n+1)}$. Hence if $x+(s,t) \in C_{m+1}$ then
\[|t|\geq 1/2^{6(n+1)}+1/2^{6(m+1)}.\]
Consequently if $z \in B(x+(s,t),1/2^{6(m+2)}) \subset H_{m+1}$ then
\[d(z,R)\geq 1/2^{6(m+1)}-1/2^{6(m+2)}.\]

On the other hand, suppose $x+(s,0) \notin \overline{H_{m}}$. Then it follows from the fact $x+(s,0)\in L_{m+1}^{k}$ and Definition \ref{set} that
\[d(x+(s,0),H_{m})\geq 1/2^{6(m+1)}.\]
Further, using both parts of the claim for $m$,
\[d(x+(s,t),H_{m})\geq 1/2^{6(m+1)}\]
for all $|t|\leq 1/2^{6(n+1)}$. Hence $x+(s,t) \in C_{m+1}$ and
\[B(x+(s,t),1/2^{6(m+2)})\subset H_{m+1}\]
for all $|t|\leq 1/2^{6(n+1)}$. 

The above analysis of $H_{m+1}\setminus H_{m}$, together with the validity of the claim for $m$, implies the claim holds for $m+1$.

Hence by induction the claim holds for all $n+1\leq m \leq n+N$.
\end{proof}

Suppose $|t|\leq 1/2^{6(n+1)}$. Then, since $x\notin H_{n+N}$, applying the claim with $m=n+N$ gives $x+tv\notin H_{n+N}$. Applying Lemma \ref{close} shows there is $y \in P$ such that
\[\|x+tv-y\|<1/2^{6(n+N)-1}.\]
That is,
\[B(x+tv,1/2^{6(n+N)-1})\cap P \neq \varnothing.\]
Hence, comparing this with equation (\ref{rho}),
\[\rho /2^{6(n+1)} \leq  1/2^{6(n+N)-1}.\]
Since $N$ could be chosen to be arbitrarily large this contradicts the assumption $\rho>0$. Hence $P$ cannot be directionally lower porous at $x$ in the direction $v$.
\end{proof}

\begin{theorem}
The set $P$ is not $\sigma$-directionally lower porous.
\end{theorem}

\begin{proof}
The set
\[G=\bigcap_{v \in D} \bigcap_{N=1}^{\infty} G_{v,N}\]
is, by Proposition \ref{goodset}, a countable intersection of dense open subsets of the closed set $P$. Hence, by Baire's theorem, $G$ is dense in $P$ so it suffices, by Proposition \ref{sufficient}, to see $P$ is not directionally lower porous at any point of $G$.

Suppose $x \in G$. Then, by Proposition \ref{ndlp}, $P$ is not directionally lower porous at $x$ in any direction belonging to $D$. Since $D$ is a dense set of directions it follows, by an easy approximation argument, that $P$ is not directionally lower porous at $x$ in any direction.
\end{proof}

\end{document}